\documentclass[11pt, a4paper]{amsart}
\usepackage{a4wide}
\usepackage[
backref=page,colorlinks,bookmarksopen=true]{hyperref}
\usepackage[latin1]{inputenc}  
\usepackage[T1]{fontenc}       

\usepackage[english]{babel}

\usepackage{amssymb}
\usepackage{latexsym}
\usepackage{mathrsfs}
\usepackage{xspace}
\usepackage{graphics, graphicx, epsfig}

\usepackage[all, cmtip]{xy}

\newtheorem{prop}{Proposition}[section]
\newtheorem{thm}[prop]{Theorem}
\newtheorem*{thm*}{Theorem}

\newtheorem*{addendum*}{Addendum}
\newtheorem{cor}[prop]{Corollary}
\newtheorem{lem}[prop]{Lemma}

\newtheorem*{convention*}{Convention}
\theoremstyle{definition}
\newtheorem*{defn*}{Definition}

\newtheorem*{scholium*}{Scholium}
\theoremstyle{remark}

\newtheorem*{example*}{Example}
\numberwithin{equation}{section}

\newcommand{\vareps}{\varepsilon}

\newcommand{\AAA}{\mathbb{A}}

\newcommand{\RR}{\mathbf{R}}

\newcommand{\g}{\mathfrak{g}}
\newcommand{\tangle}[2]
{\angle_\mathrm{T}(#1,#2)}
\newcommand{\aangle}[3]
{\angle_{#1}(#2,#3)}
\newcommand{\cangle}[3]
{\overline{\angle}_{#1}(#2,#3)}

\def\max{\mathop{\mathrm{max}}\nolimits}
\def\Op{\mathop{\mathrm{Opp}}\nolimits}
\begin{document}
\title[On the distortion of twin building lattices]{On the distortion of twin building lattices}

\author[Pierre-Emmanuel Caprace]{Pierre-Emmanuel Caprace*}
\address{Universit\'e Catholique de Louvain\\ Chemin du Cyclotron 2\\ 1348 Louvain-la-Neuve\\ Belgium}
\email{pierre-emmanuel.caprace@uclouvain.be}
\thanks{*Supported by the Fund for Scientific Research--F.N.R.S., Belgium}
\author[Bertrand R\'emy]{Bertrand R\'emy**}
\address{Universit\'e de Lyon, Universit\' e Lyon 1\\
Institut Camille Jordan\\
UMR 5208 du CNRS\\
B\^atiment J. Braconnier\\
43 blvd du 11 novembre 1918\\
F-69622 Villeurbanne Cedex -- France
}
\email{remy@math.univ-lyon1.fr}
\thanks{**Supported in part by ANR project GGPG: G\'eom\'etrie et Probabilit\'es dans les Groupes}
\date{July 20, 2009}
\begin{abstract}
We show that twin building lattices are undistorted in their ambient
group; equivalently, the orbit map  of the lattice to the product of
the associated twin buildings is a quasi-isometric embedding. As a
consequence, we provide an estimate of the quasi-flat rank of these
lattices, which implies that there are infinitely many
quasi-isometry classes of finitely presented simple groups. In an
appendix, we describe how non-distortion of lattices is related to
the integrability of the structural cocycle.
\end{abstract}

\maketitle

\section{Introduction}

\subsection{Distortion} Let $G$ be a locally compact group and $\Gamma < G$ be a finitely generated lattice.
Then $G$ is compactly generated \cite[Lemma~2.12]{CaMo} and
therefore both $G$ and $\Gamma$ admit word metrics, which are well
defined up to quasi-isometry. It is a natural question to understand
the relation between the word metric of $\Gamma$ and the restriction
to $\Gamma$ of the word metric on $G$.

In order to address this issue, let us fix some compact generating
set $\widehat{\Sigma}$ in $G$ and denote by $\| g
\|_{\widehat{\Sigma}}$ the word length of an element $g \in G$ with
respect to $\widehat{\Sigma}$; we denote by $d_{\widehat{\Sigma}}$
the associated word metric. Similarly, we fix a finite generating
set $\Sigma$ for $\Gamma$ and denote by $| \gamma |_\Sigma$ the word
length of an element $\gamma \in \Gamma$ with respect to $\Sigma$,
and by $d_\Sigma$ the associated word metric. The lattice $\Gamma$
is called \textbf{undistorted} in $G$ if $d_\Sigma$ is
quasi-isometric to the restriction of $d_\Sigma$ to $\Gamma$. The
condition amounts to saying that the inclusion of $\Gamma$ in $G $
defines a quasi-isometric embedding from the metric space
$(\Gamma,d_\Sigma)$ to the metric space $(G,d_{\widehat{\Sigma}})$.

As is well-known, any cocompact lattice is undistorted: this follows
from the \v{S}varc--Milnor Lemma
\cite[Proposition~I.8.19]{Bridson-Haefliger}. The question of
distortion thus centres around non-uniform lattices. The main result
of \cite{LMR_nondistorsion} is that if $G$ is a product of
higher-rank semi-simple algebraic groups over local fields
(Archimedean or not), then any lattice of $G$ is undistorted. This
relies on the deep arithmeticity theorems due to Margulis in
characteristic 0 and Venkataramana in positive characteristic, and
on a detailed analysis of the distortion of unipotent subgroups.

Besides the higher-rank lattices in semi-simple groups, a class of
non-uniform lattices that has attracted some attention in recent
years are the so-called Kac--Moody lattices (see \cite{RemCRAS} or
\cite{CarGar}). A more general class of lattices is that of twin
building lattices \cite{CaRe}: a \textbf{twin building lattice} is
an irreducible lattice $\Gamma < G = G_+ \times G_-$ in a product of
two groups $G_+$ and $G_-$ acting strongly transitively on (locally
finite) buildings $X_+$ and $X_-$ respectively, and such that
$\Gamma$ preserves a twinning between $X_+$ and $X_-$. Recall that
$\Gamma$ is then finitely generated and that, in this general
context, \textbf{irreducible}~means that each of the projections of
$\Gamma$ to $G_\pm$ is dense.

\begin{thm}
\label{thm:KM:non-distortion}
Any twin building lattice $\Gamma < G_+ \times G_-$ is undistorted.
\end{thm}

It should be noted that each individual group $G_+$ or $G_-$ also
possesses non-uniform lattices, obtained for instance by
intersecting $\Gamma$ with a compact open subgroup (\emph{e.g.}, a
facet stabilizer) of $G_-$ or $G_+$, respectively. Other non-uniform
lattices have been constructed by R.~Gramlich and
B.~M\"uhlherr~\cite{GramlichMuehlherr}. We emphasize that, beyond
the affine case (\emph{i.e.} when $G_+$ is a semi-simple group over
a local function field), a non-unifiorm lattice in a single
irreducible factor $G_+$ (or $G_-$) should be expected to be
automatically \emph{distorted} (see Section~\ref{sec:remark} below).

\subsection{Quasi-isometry classes} Non-distortion of a lattice $\Gamma$ in $G$ relates the intrinsic geometry of $\Gamma$ to the geometry of $G$.
In the case of twin building lattices, the latter geometry is
(quasi-isometrically) equivalent to the geometry of the product
building $X_+ \times X_-$ on which $G$ acts cocompactly.
Non-distortion is especially relevant when studying quasi-isometric
rigidity of $\Gamma$ (which is still an open problem). As a
consequence of Theorem~\ref{thm:KM:non-distortion}, we can estimate
a quasi-isometric invariant of a twin building lattice $\Gamma$ for
$X_+ \times X_-$, namely the maximal dimension of
quasi-isometrically embedded flat subspaces into $(\Gamma,
d_\Sigma)$. This rank is bounded from below by the maximal dimension
of an isometrically embedded flat in $X_\pm$ and from above by twice
the same quantity (\ref{ss - ranks}); furthermore, thanks to
D.~Krammer's thesis \cite{Krammer}, this metric rank of $X_\pm$ can
be computed concretely by means of the Coxeter diagram of the Weyl
group of $X_\pm$. This enables us to draw the following
group-theoretic consequence.

\begin{cor}
\label{cor:QIclasses}
There exist infinitely many pairwise non-quasi-isometric finitely presented simple groups.
\end{cor}

This corollary may also be deduced from the work of J.~Dymara and Th.~Schick \cite{DymaraSchick}, which gives an estimate of another quasi-isometry invariant for twin building lattices, namely the \emph{asymptotic dimension}.

Any finite simple group is of course quasi-isometric to the trivial
group. Moreover any finitely presented simple group constructed by
M.~Burger and Sh.~Mozes \cite{Burger-Mozes2} is quasi-isometric to
the product of free groups $F_2 \times F_2$; this is due to
\cite{Papasoglu} and to the fact that the latter groups are
constructed as suitable (torsion-free) uniform lattices in products
of trees. Furthermore, concerning the finitely presented simple
groups constructed by G.~Higman and R.~Thompson \cite{Higman}, as
well as their avatars in \cite{Rover}, \cite{Brin},
\cite{BrownSimple} and \cite{ScottSimple}, we are not aware of a
classification up to quasi-isometry as of today. However some
results seem to indicate that many of them might be quasi-isometric
to one another, compare \emph{e.g.}~\cite{BurilloClearyStein}.

\subsection{Integrability of the structural cocycle}
Non-distortion of lattices is also relevant, in a more subtle way, to the theory of unitary representations and its applications.
More precisely, given a lattice $\Gamma < G$ and a unitary $\Gamma$-representation $\pi$, one considers the induced $G$-representation $\mathrm{Ind}_\Gamma^G \pi$.
For rigidity questions (at least) and also because the structure of $G$ is richer than that of $\Gamma$, it is desirable that the cocycles of $\Gamma$ with coefficients in $\pi$ extend to continuous cocycles of $G$ with coefficients in $\mathrm{Ind}_\Gamma^G \pi$.
As explained in \cite[Proposition~1.11]{Shalom00}, a sufficient condition for this to hold is that $\Gamma$ be \textbf{square-integrable}.
By definition, for any $p \in [1;\infty)$ it is said that $\Gamma$ (or more precisely the inclusion $\Gamma < G$) is \textbf{$p$-integrable}~if there is a Borel fundamental domain $\Omega \subset G$ for $G/\Gamma$ such that, for each $g \in G$, we have:

$$\int_\Omega \bigl( | \alpha(g, h) |_\Sigma \bigr)^p \mathrm{d}h < \infty,$$

\noindent where $\alpha : G \times \Omega \to \Gamma$ is the induction cocycle defined by $\alpha(g, h) = \gamma \Leftrightarrow g h \gamma \in \Omega$. Mimicking Y.~Shalom's arguments in \cite[\S2]{Shalom00},  the following statement will be established in an appendix (with the above notation for generating sets).

\begin{thm}
\label{th:InductionCocycle}
Let $G$ be a totally disconnected locally compact group and let $\Gamma < G$ be a finitely generated lattice.
Assume there is a Borel fundamental domain $\Omega \subset G$ for $G/\Gamma$ such that for some $p \in [1;\infty)$ we have:
$$\int_\Omega \bigl( \| h \|_{\widehat{\Sigma}} \bigr)^p \mathrm{d}h < \infty.$$
Then, if $\Gamma$ is non-distorted, it is $p$-integrable
\end{thm}

For $S$-arithmetic groups, the existence of fundamental domains satisfying the condition of Theorem~\ref{th:InductionCocycle} is established in \cite[Proposition~VIII.1.2]{Margulis} by means of Siegel domains.
As we shall see, in the case of twin building lattices the condition is straightforward to check once a fundamental domain provided by the specific combinatorial properties of these lattices is used.
In particular, combining Theorem~\ref{thm:KM:non-distortion} with Theorem~\ref{th:InductionCocycle}, we recover the main result of~\cite{Remy:Integrable}.
We finish by mentioning that square-integrability of lattices is also relevant for lifting $\Gamma$-actions to $G$-actions in geometric situations which are much more general than unitary actions on Hilbert spaces, see \cite{Monod_superrigid} and \cite{Gelander-Karlsson-Margulis}.

\medskip

In order to always start from the same situation, in the above
introduction we stated results exclusively dealing with group
inclusions. The proof of the non-distortion statement is of
geometric nature: we prove that a twin building lattice is
non-distorted in the product of the two buildings with which it is
associated.

\medskip
This article is written as follows. Section~\ref{sec:prel} consists
of preliminaries. Section~\ref{s - non-distortion} provides the
aforementioned geometric proof of non-distortion and deals with the
various metric notions of ranks that can be better understood thanks
to non-distortion; we apply this to quasi-isometry classes of
finitely generated simple groups. Appendix~\ref{sec:integrability}
is independent of the previous setting of twin building lattices and
establishes a relationship between non-distortion and
square-integrability of lattices in general totally disconnected
locally compact groups.

\section{Lifting galleries from the buildings to the lattice}\label{sec:prel}

We refer to \cite{AbrBrown} for basic definitions and facts on
buildings and twinnings, and to \cite{CaRe} for twin building
lattices. In this preliminary section, we merely fix the notation
and recall one basic fact on twin buildings which plays a key role
at different places in this paper.

Let $X = (X_+, X_-)$ be a twin building with Weyl group $W$ associated to a group $\Gamma$ admitting a root group datum. In particular $\Gamma$ acts strongly transitively on $X$. We let $d_{X_+}$  (resp.~$d_{X_-}$) denote the combinatorial distance on the set of chambers of $X_+$ (resp.~$X_-$). We further denote by $S$ the canonical generating set of $W$ and by $\Op(X)$ the set of pairs of opposite chambers of $X$. Throughout the paper, we fix a base pair $(c_+, c_-)\in \Op(X)$ and call it the \textbf{fundamental opposite pair} of chambers. Two opposite pairs $(x_+, x_-)$ and $(y_+, y_-) \in \Op(X)$ are called \textbf{adjacent} if there is some $s \in S$ such that $x_+$ is $s$-adjacent to $y_+$ and $x_-$ is $s$-adjacent to $y_-$. Recall that an opposite pair $x \in \Op(X)$ is contained in unique twin apartment, which we shall denote by $\AAA(x) = \AAA(x_+, x_-)$. The positive (resp. negative) half of $\AAA(x)$ is denoted by $\AAA(x)_+$ (resp.~$\AAA(x)_-$).

\medskip
The following key property is well known to the experts, and appear implicitly in the proof of Proposition~5 in \cite{Tits88/89}.
\begin{lem}\label{lem:key}
Let $\vareps \in \{+, -\}$. Given any gallery $(x_0, x_1, \dots, x_n)$ in $X_\vareps$ and any chamber $y_0 \in X_{-\vareps}$ opposite $x_0$, there exists a gallery $(y_0, y_1, \dots, y_n)$ in $X_{-\vareps}$ such that the following hold for all $i = 1, \dots, n$:
\begin{enumerate}
\item $(x_i, y_i)\in \Op(X)$;
\item $(x_i, y_i)$ is adjacent to $(x_{i-1}, y_{i-1})$;
\item  $y_i$ belongs to the twin apartment $\AAA(x_0, y_0)$.
\end{enumerate}
\end{lem}

\begin{proof}
The desired gallery is constructed inductively as follows. Let
$i>0$. If $y_{i-1}$ is opposite $x_i$, then set $y_i = y_{i-1}$.
Otherwise the codistance $\delta^*(x_i, y_{i-1})$ is an element $s
\in S$ and there is a unique chamber in the twin apartment
$\AAA(x_0, y_0)$ which is $s$-adjacent to $y_{i-1}$. Define $y_i$ to
be that chamber. It follows from the axioms of a twinning that $y_i$
is opposite $x_i$. The gallery $(y_0, y_1, \dots, y_n)$ constructed
in this way satisfies all the desired properties.
\end{proof}

\section{Non-distortion of twin building lattices}
\label{s - non-distortion}

In this section, we show that a twin building lattice is
non-distorted for its natural diagonal action on its two twinned
building. The arguments are elementary and  use the basic
combinatorial geometry of buildings.

\subsection{An adapted generating system}
\label{ss - generating system}

Let $\Sigma$ denote the subset of $\Gamma$ consisting of those elements $\gamma$ such that $(\gamma.c_+, \gamma.c_-)$ is adjacent to $(c_+, c_-)$, where $(c_+, c_-) \in \Op(X)$ denotes the fundamental opposite pair. Notice that
$$\max\{ d_{X_+}(c_+, \gamma.c_+); d_{X_-}(c_-, \gamma.c_-)\}\leqslant 1$$
for all $\gamma \in S$.

The graph structure on $\Op(X)$ induced by the aforementioned adjacency relation is isomorphic to the Cayley graph associated to the pair $(\Gamma, \Sigma)$. Lemma~\ref{lem:key} readily implies that this graph is connected. Thus $\Sigma$ is a generating set for $\Gamma$.

\begin{lem}
\label{lem - adjacent case}
Let $z=(z_+,z_-)$ be a pair of opposite chambers such that
$$\max\{ d_{X_+}(c_+, z_+); d_{X_-}(c_-, z_-)\}\leqslant 1.$$
Then there exists $\sigma \in \Sigma$ such that $\sigma.z=c$.
\end{lem}

\begin{proof}
It is enough to deal with the case when $\max\{ d_{X_+}(c_+, z_+); d_{X_-}(c_-, z_-)\} = 1$.

If both $z_-$ and $z_+$ belong to the twin apartment $\mathbb{A} =
\mathbb{A}_- \sqcup \mathbb{A}_+$, we can write $z_+=w_+.c_+$ and
$z_-=w_-.c_-$ for $w_\pm \in W$ uniquely defined by $z_\pm$. Since
$z_-$ and $z_+$ are assumed to be opposite, the codistance
$\delta^*(z_-,z_+)$ is by definition equal to $1_W$. Since the
diagonal $\Gamma$-action on $X_- \times X_+$ preserves codistances,
we deduce that $w_+=w_-$. At last since $\max\{ d_{X_+}(c_+, z_+);
d_{X_-}(c_-, z_-)\} = 1$, we deduce that there exists a canonical
reflection $s \in S$ such that $w_\pm=s$ and this reflection is
represented by an element $n_s \in {\rm Stab}_\Gamma(\mathbb{A})$;
we clearly have $n_s \in \Sigma$.

We henceforth deal with the case when at least one of the elements
$z_\pm$ does not lie in $\mathbb{A}$. Up to switching signs, we may
-- and shall -- assume that $z_- \not\in \mathbb{A}_-$. Let $s$ be
the canonical reflection such that $z_-$ is $s$-adjacent to $c_-$.
By the Moufang property, the group $U_{-\alpha_s}$ acts simply
transitively on the chambers  $\neq c_-$ which are $s$-adjacent to
$c_-$. By conjugating by an element $n_s$ as above and since $z_-
\neq s.c_-$ (because $z_- \not\in \mathbb{A}_-$), we conclude that
there exists $u_+ \in U_{\alpha_s} \setminus \{Ê1 \}$ such that
$u_+.z_-=c_-$. Moreover $u_+$ stabilizes $c_+$ so the chamber
$u_+.z_+$ is adjacent to $c_+$.

If $u_+.z_+ \in \mathbb{A}_+$, then since the $\Gamma$-action preserves the codistance, the chamber $u_+.z_+ \in \mathbb{A}_+$ is the unique chamber in $\mathbb{A}$ which is opposite $c_-=u_+.z_-$, namely $c_+$; we are thus done in this case because we clearly have $u_+ \in \Sigma$.

We finish by considering the case when $u_+.z_+ \not\in \mathbb{A}_+$.
Then there exists some canonical reflection $t \in S$ such that $u_+.z_+$ is $t$-adjacent to $c_+$ and we can find similarly an element $u_- \in U_{-\alpha_t} \setminus \{Ê1 \}$ such that $u_-.(u_+.z_+)=c_+$.
Setting $\sigma = u_-u_+$, we obtain an element of $\Gamma$ sending $z_\pm$ to $c_\pm$.
Since the $\Gamma$-action preserves each adjacency relation, hence the combinatorial distances, we have $\sigma \in \Sigma$ because
$d_{X_-}(c_-,\sigma.c_-) = d_{X_-}(u_-^{-1}.c_-, u_+.c_-) = d_{X_-}(c_-, u_+.c_-) = 1$ and
$d_{X_+}(c_+,\sigma.c_+) = d_{X_+}(c_+, u_-.c_+) = 1$.
\end{proof}

\subsection{Proof of non-distortion}
\label{ss - proof}
We define the combinatorial distance $d_X$ of the chamber set of $X$ by
$$d_X\big((x_+,x_-), (y_+,y_-)\big) = d_{X_+}(x_+, y_+) + d_{X_-}(x_-, y_-).$$

Since the $G$-action on $X$ is cocompact, it follows from the \v{S}varc--Milnor lemma
\cite[Proposition~I.8.19]{Bridson-Haefliger} that $G$ is quasi-isometric to $X$. Hence
Theorem~\ref{thm:KM:non-distortion} is an immediate consequence of the following.

\begin{prop}
\label{prop:OrbitMap}
Let $\Gamma < G = G_+ \times G_-$ be a twin building lattice associated with the twin building $X = X_+ \times X_-$ and let $c = (c_+, c_-) \in X$ be a pair of opposite chambers.
Then for each $\gamma \in \Gamma$, we have:
$${1 \over 2}  d_X(c, \gamma.c) \, \leqslant \, |\gamma |_\Sigma \, \leqslant \, 2  d_X(c, \gamma.c).$$
\end{prop}

\begin{proof}
[Proof of Proposition~\ref{prop:OrbitMap}]
Writing $\gamma \in \Gamma$ as a product of $|\gamma |_\Sigma$ elements of the generating set $\Sigma$ and using triangle inequalities, we obtain
$$d_X(c, \gamma.c) \leqslant 2  |\gamma |_\Sigma$$
by the definition of $d_X$ and of $\Sigma$.

\smallskip

It remains to prove the other inequality, which says that $\Gamma$-orbits spread enough in $X$.
We set $x = (x_+, x_-) = \gamma^{-1}.c$.
Let us pick a minimal gallery in $X_-$, from $x_-$ to $c_-$.
Using auxiliary positive chambers, one opposite for each chamber of the latter gallery, a repeated use of Lemma \ref{lem - adjacent case} shows that there exists $\gamma_- \in \Gamma$ such that $\gamma_-.x_- = c_-$ and

\medskip

\hspace{1cm}$(*)$ \quad $|\gamma_- |_\Sigma \, \leqslant \, d_{X_-}(c_-, x_-)$.

\medskip

\noindent Moreover as in the first paragraph, we have:

\medskip

\hspace{1cm}$(**)$ \quad $d_{X_+}(c_+, \gamma_-.c_+) \leqslant |\gamma_-|_\Sigma$,

\medskip

\noindent by the definition of $\Sigma$. We deduce:

$$\begin{array}{rcl}
d_{X_+}(c_+, \gamma_-.x_+) & \leqslant
& d_{X_+}(c_+, \gamma_-.c_+) + d_{X_+}(\gamma_-.c_+,  \gamma_-.x_+)\\
& \leqslant & |\gamma_-|_\Sigma + d_{X_+}(c_+, x_+)\\
& \leqslant & d_{X_-}(c_-, x_-) + d_{X_+}(c_+, x_+),
\end{array}$$

\noindent successively by the triangle inequality, by $(**)$ and the fact the $\Gamma$-action is isometric for the combinatorial distances on chambers, and by $(*)$.
Therefore, by definition of $d_X$, we already have:

\medskip

\hspace{1cm}$(***)$ \quad $d_{X_+}(c_+, \gamma_-.x_+) \leqslant d_{X}(c, x)$.

\smallskip

We now construct a suitable element $\gamma_+ \in \Gamma$ such that $\gamma_+.x_+ = c_+$ and $\gamma_+.c_- = c_-$.
Let $\gamma_-. x_+ = z_0, z_1, \dots, z_k = c_+$ be a minimal gallery in $X_+$ from $\gamma_-. x_+$ to $c_+$.
Let $\mathbb{A} = \mathbb{A}_+ \sqcup \mathbb{A}_-$ be the twin apartment defined by the opposite pair $c=(c_+, c_-)$.
Let  $c_0=c_-, c_1, \dots, c_k$ be the gallery contained in $\mathbb{A}_+$ and associated to $ z_0, z_1, \dots, z_k = c_+$ as in Lemma~\ref{lem:key}.  Notice that, since $c_k$ is opposite $z_k = c_+$ and since $c_-$ is the \emph{unique}~chamber of $\mathbb{A}_-$ opposite $c_+$, we have $c_k = c_-$.

\smallskip

By Lemma \ref{lem - adjacent case}, there exists $\sigma_1 \in \Sigma$ such that $\sigma_1.z_{k-1} = z_k$ and $\sigma_1.c_{k-1} = c_k$.
Moreover a straightforward inductive argument yields for each $i \in \{1, \dots, k\}$ an element $\sigma_i \in \Sigma$ such that
$\sigma_i \sigma_{i-1} \dots \sigma_1.z_{k-i} = z_k$ and $\sigma_i \sigma_{i-1} \dots \sigma_1.c_{k-i} = c_k$.
Let now $\gamma_+ = \sigma_k \dots \sigma_1$, so that $|\gamma_+|_\Sigma \leqslant k = d_{X_+}(c_+,\gamma_-. x_+)$.
By construction, we have $\gamma_+.(\gamma_-.x_+) = c_+$ and $\gamma_+.c_- = c_-$, that is $(\gamma_+\gamma_-).x =c$.
Therefore $(\gamma_+ \gamma_- \gamma^{-1}).c = c$ and hence there is $\sigma \in \Sigma$ such that
$\gamma= \sigma \gamma_+ \gamma_-$.
In fact, since $\sigma$ fixes $c$, it follows that $\sigma\sigma' \in \Sigma$ for each $\sigma' \in \Sigma$.
Upon replacing $\sigma_k$ by $\sigma\sigma_k$, we may -- and shall -- assume that $\gamma= \gamma_+ \gamma_-$.
Therefore we have:

$$\begin{array}{rcl}
|\gamma|_\Sigma &\leqslant &|\gamma_+|_\Sigma \, + \, | \gamma_-|_\Sigma\\
& \leqslant & d_{X_+}(c_+, \gamma_-.x_+) + d_{X_-}(c_-, x_-),\\
\end{array}$$
the last inequality coming from $|\gamma_+|_\Sigma \leqslant k = d_{X_+}(c_+,\gamma_-. x_+)$ and $(*)$ above.
By $(***)$ and the definition of $d_X$, this finally provides $|\gamma|_\Sigma \leqslant 2 \cdot d_X(c,\gamma.c)$, which finishes the proof.
\end{proof}

\subsection{A remark on distortion of lattices in rank one groups}\label{sec:remark}
Let $G = G_+ \times G_-$ be product of two totally disconnected
locally compact groups, let $\pi_\pm :G \to G_\pm$ denote the
canonical projections and let $\Gamma < G$ be a finitely generated
lattice. Assume that  $\overline{\pi_-(\Gamma)}$ is cocompact in
$G_-$ (this is automatic for example if  $\Gamma$ is irreducible).
Let also $U_- < G_-$ be a compact open subgroup and set  $\Gamma_- =
\Gamma \cap (G_+ \times U_-)$. Then the projection of $\Gamma_-$ to
$G_+$ is a lattice, and it is straightforward to verify that,
\emph{if  $\Gamma_-$ is finitely generated and undistorted in $G_-$,
then $\Gamma$ is undistorted in $G$.}

We emphasize however that, in the case of twin building lattices, the lattice $\Gamma_-$ should not be expected to be undistorted in $G_-$ beyond the affine case (which corresponds to the classical case of arithmetic lattices in semi-simple groups over local function fields). Indeed, a typical non-affine case is when $G_+$ and $G_-$ are Gromov hyperbolic (equivalently, the Weyl group is Gromov hyperbolic or, still equivalently, each of the buildings $X_+$ and $X_-$ are Gromov hyperbolic). Then a non-uniform lattice in $G_+$ is always distorted, as follows from the following.

\begin{lem}\label{lem:RankOne}
Let $G$ be a compactly generated Gromov hyperbolic totally disconnected locally compact group and $\Gamma < G$ be a finitely generated lattice. Then the following assertions are equivalent.
\begin{enumerate}
\item $\Gamma$ is a uniform lattice.

\item $\Gamma$ is undistorted in $G$.

\item $\Gamma$ is a Gromov hyperbolic group.
\end{enumerate}
\end{lem}

\begin{proof}
(i) $\Rightarrow$ (ii)
Follows from the \v Svarc--Milnor Lemma.

\medskip \noindent
(ii) $\Rightarrow$ (iii)
Follows from the well-known fact that a quasi-isometrically embedded subgroup of a Gromov hyperbolic group is quasi-convex.

\medskip \noindent
(iii) $\Rightarrow$ (i)
By Serre's covolume formula (see \cite{Serre}) a non-uniform lattice in a totally disconnected locally compact group possesses finite subgroups of arbitrary large order, and can therefore not be Gromov hyperbolic.
\end{proof}

\subsection{Various notions of rank}
\label{ss - ranks}
As a consequence of Theorem~\ref{thm:KM:non-distortion}, we obtain the following estimate for one of the most basic quasi-isometric invariants attached to a finitely generated group.

\begin{cor}
\label{cor:FlatRank}
Let $\Gamma  < G = G_+ \times G_-$ be a twin building lattice with finite symmetric generating subset $\Sigma$.
Let $r$ denote the quasi-flat rank of $(\Gamma, d_\Sigma)$ and let $R$ denote the flat rank of the building $X_\pm$.
Then we have: $R \leqslant r \leqslant 2R$.
\end{cor}

Recall that by definition, the \textbf{flat rank} (resp.
\textbf{quasi-flat rank}) of a metric space is the maximal rank of a
flat (resp. quasi-flat), \emph{i.e.} an isometrically embedded
(resp. quasi-isometrically embedded) copy of $\RR^n$. By \cite{CaHa}
the flat rank of a building coincides with the maximal rank of a
free Abelian subgroup of its Weyl group $W$, and this quantity may
be computed explicitly in terms of the Coxeter diagram of $W$, see
\cite[Theorem~6.8.3]{Krammer}.

\begin{proof}
[Proof of Corollary~\ref{cor:FlatRank}]
Let us first prove $r
\leqslant 2R$. Let $\varphi : ({\bf R}^r, d_{\rm eucl}) \to (\Gamma,
d_\Sigma)$ denote a quasi-isometric embedding of a Euclidean space
in the Cayley graph of $\Gamma$. With the notation of
Proposition~\ref{prop:OrbitMap}, we know that the orbit map
$\omega_c : \Gamma \to X_+ \times X_-$ defined by $\gamma \mapsto
\gamma.c$ is a quasi-isometric embedding. Therefore the composed map
$\omega_c \circ \varphi : ({\bf R}^r, d_{\rm eucl}) \to X_+ \times
X_-$ is a quasi-isometric embedding. By \cite[Theorem C]{Kleiner},
this implies the existence of flats of dimension $r$ in the product
of two spaces of flat rank $R$; hence $r \leqslant 2R$.

\medskip

We now turn to the inequality $R \leqslant r$. As mentioned above,
it is shown in \cite{CaHa} the flat rank of a building coincides
with the flat rank of any of its apartment. Since the standard twin
apartment is contained in the image of $\Gamma$ under the orbit map
$\Gamma \to X_+ \times X_-$, the desired inequality follows directly
from the non-distortion of $\Gamma$ established in
Proposition~\ref{prop:OrbitMap}.
\end{proof}

Note that another notion of rank, relevant to G. Willis' general theory of totally disconnected locally compact groups, is discussed for the full automorphism groups $G_\pm = {\rm Aut}(X_\pm)$ in \cite{BRW}, and turns out to coincide with the above notions of rank.

\begin{proof}
[Proof of Corollary~\ref{cor:QIclasses}] Since there exist twin
buildings of arbitrary flat rank (choose for instance Dynkin
diagrams such that the associated Coxeter diagram contains more and
more commuting $\widetilde A_2$-diagrams), we deduce that twin
building lattices fall into infinitely many quasi-isometry classes.
This observation may be combined with the simplicity theorem from
\cite{CaRe} to yield the desired result.
\end{proof}

\appendix
\section{Integrability of undistorted lattices}\label{sec:integrability}

In this section, we give up the specific setting of twin building
lattices and provide a simple condition ensuring that non-distorted
finitely generated lattices in totally disconnected groups are
square-integrable.

\subsection{Schreier graphs and lattice actions}
\label{ss - Schreier graphs}
Let us a consider a totally disconnected, locally compact group $G$.
As before we assume that $G$ contains a finitely generated lattice, say $\Gamma$, which implies that $G$ is compactly generated \cite[Lemma~2.12]{CaMo}.
By \cite[III.4.6, Corollaire 1]{BourbakiTGI}, we know that $G$ contains a compact open subgroup, say $U$.
Let $C$ be a compact generating subset of $G$ which, upon replacing $C$ by $C \cup C^{-1}$, we may -- and shall -- assume to be symmetric: $C=C^{-1}$.
We set $\widehat{\Sigma} = UCU$, which is still a symmetric generating set for $G$.

We now introduce the \textbf{Schreier graph}~$\mathfrak{g}_{U,\widehat{\Sigma}}$, or simply $\mathfrak{g}$, associated to the above choices.
It is the graph whose set of vertices is the discrete set $G/U$, which is countable whenever $G$ is $\sigma$-compact.
Two distinct vertices $gU$ and $hU$ are connected by an edge if, and only if, we have $g^{-1}h \in \widehat{\Sigma}$ \cite[\S11.3]{Monod_LN}.
The natural $G$-action on $\mathfrak{g}$ by left translation is proper, and it is isometric whenever we endow $\mathfrak{g}$ with the metric $d_\g$ for which all edges have length 1.
We view the identity class $1_GU$ as a base vertex of the graph $\mathfrak{g}$, which we denote by $v_0$.

Denoting by $\| \cdot \|_{\widehat{\Sigma}}$ the word metric on $G$ attached to $\widehat{\Sigma}$, we have: $\|g\|_{\widehat{\Sigma}} = d_{\widehat{\Sigma}} (1_G, g)$ for any $g \in G$.
Notice that the generating set $\widehat{\Sigma}$ of $G$ consists by definition  of those elements $g \in G$ such that $d_\g(v_0, g.v_0) \leqslant 1$.
In particular, for all $g, h \in G$, we have:
$$d_\g (g.v_0, h.v_0) \leqslant d_{\widehat{\Sigma}}(g, h) \leqslant d_\g (g.v_0, h.v_0) +1.$$
Moreover $d_\g (g.v_0, h.v_0) = d_{\widehat{\Sigma}}(g, h)$ whenever $g.v_0 \neq h.v_0$.

In the present setting, using again \cite[III.4.6, Corollaire
1]{BourbakiTGI} and the discreteness of the $\Gamma$-action, we may
-- and shall -- work with a Schreier graph $\mathfrak{g}$ defined by
a compact open subgroup $U$ small enough to satisfy $\Gamma \cap U=
\{Ê1_G \}$. Thus we have:
$${\rm Stab}_\Gamma(v_0) = \Gamma \cap U = \{Ê1_G \}.$$

Let $\mathscr{V} = \{v_0, v_1, \dots\}$ be a set of representatives for the $\Gamma$-orbits of vertices.
The element $v_0$ is the previous one, and for each $i > 0$, we choose $v_i$ in such a way that $d_\g(v_i,v_0) \leqslant d_\g(v_i, \gamma.v_0)$ for all $\gamma
\in \Gamma$; this is possible because the distance $d_\g$ takes integral values.
We set $g_0 = 1$; for each $i > 0$, since the $G$-action on the vertices of $\mathfrak{g}$ is transitive, there exists $g_i \in G$ such that $g_i.v_i = v_0$.
Thus for any $g \in G$ there exists $j \geqslant 0$ such that $g.v_0 \in \Gamma.v_i$, which provides the partition:
$$G = \bigsqcup_{j \geqslant 0}Ê\Gamma g_j^{-1}U.$$

Furthermore, for each $i \geqslant 0$, we choose a Borel subset $V_i \subset U$ which is a section of the right $U$-orbit map $U \to \Gamma \setminus (\Gamma g_i^{-1} U)$ defined by $u \mapsto \Gamma g_i^{-1}u$.
Setting $F_i = g_i^{-1}V_ig_i$, we obtain a subset $F_i$ of ${\rm Stab}_G(v_i)$ such that
$$\mathscr{F} = \bigsqcup_{i \geqslant 0} F_{v_i}g_i^{-1}$$
is a Borel fundamental domain for $\Gamma$ in $G$.
We normalize the Haar measure on $G$ so that $\mathscr{F}$ has volume~$1$.

\subsection{Non-distortion implies square-integrability}
\label{ss - sq-int}
We can now turn to the proof of the latter implication, more precisely Theorem~\ref{th:InductionCocycle}.

\begin{proof}
[Proof of Theorem~\ref{th:InductionCocycle}]
Let $g \in G$ and $h \in \mathscr{F}$.

\smallskip

On the one hand, by definition of the induction cocycle $\alpha : G \times \mathscr{F} \to \Gamma$, the element $\alpha(g, h) = \gamma \in \Gamma$ is defined by $\gamma hg \in\mathscr{F}$.
Therefore, by construction of the fundamental domain $\mathscr{F}$, there exist $i \geqslant 0$ and $u \in F_i$ such that $\gamma hg = ug_i^{-1}$.
Let us apply the latter element to the origin $v_0$ of $\mathfrak{g}$.
We obtain $\gamma hg.v_0 = ug_i^{-1}v_0 = u.v_i$, and since $u \in F_i$ and $F_i \subset {\rm Stab}_G(v_i)$, this finally provides $\gamma hg.v_0=v_i$.
By this and the choice of $v_i$ in its $\Gamma$-orbit, we have:

\medskip

\hspace{1cm}$(\star)$ \quad $d_\mathfrak{g}(v_0,v_i) \leqslant d_\mathfrak{g}(v_0,\gamma^{-1}.v_i) = d_\mathfrak{g}(v_0, hg.v_0)$.

\smallskip

On the other hand, let $\Sigma$ be a finite symmetric generating set for $\Gamma$ and let $d_\Sigma$ be the associated word metric; we set $| \gamma |_\Sigma = d_\Sigma (1_G,\gamma)$ for $\gamma \in \Gamma$.
Since the metric spaces $(G,d_{\widehat{\Sigma}})$ and $(\mathfrak{g},d_\mathfrak{g})$ are quasi-isometric (\ref{ss - Schreier graphs}), the assumption that $\Gamma$ is undistorted is equivalent to the fact that the $\Gamma$-orbit map $\Gamma \to \g$ of $v_0$ defined by $\gamma \mapsto \gamma.v_0$
is a quasi-isometric embedding.
In particular, there exist constants $L \geqslant 1$ and $M \geqslant 0$ such that
$$|\gamma|_\Sigma \leqslant L \cdot d_\g(v_0, \gamma.v_0) +C$$
for all $\gamma \in \Gamma$.
Moreover $d_\g$ takes integer values and ${\rm Stab}_\Gamma(v_0) = \{Ê1_G \}$, so for all non-trivial $\gamma \in \Gamma$ we have: $L .d_\g(v_0, \gamma.v_0) +C \leqslant (L+C).d_\g(v_0, \gamma.v_0)$.
Therefore, upon replacing $L$ by a larger constant we may -- and shall -- assume that $C = 0$.

\medskip

Our aim is to evaluate $| \gamma |_\Sigma = |\alpha(g, h)|_\Sigma$ in terms of $\| g \|_{\widehat{\Sigma}}$ and $\| h \|_{\widehat{\Sigma}}$.
Note that $| \gamma |_\Sigma = | \gamma^{-1} |_\Sigma$ since $\Sigma$ is symmetric.

First, we deduce successively from non-distortion, from the triangle inequality inserting $\gamma^{-1}.v_i$, and from the fact that the $\Gamma$-action on $\mathfrak{g}$ is isometric, that:
$$\begin{array}{rcl}
| \gamma^{-1} |_\Sigma & \leqslant & L \cdot d_\g(v_0, \gamma^{-1}.v_0)\\
& \leqslant & L \cdot \big(d_\g(v_0, \gamma^{-1}.v_j) + d_\g(\gamma^{-1}.v_0, \gamma^{-1}.v_j)\big)\\
& \leqslant & L \cdot \big(d_\g(v_0, \gamma^{-1}.v_j) + d_\g(v_0, v_j)\big).
\end{array}$$

Then, we deduce successively from $(\star)$, from the triangle inequality inserting $h.v_0$, and from the fact that the $G$-action on $\mathfrak{g}$ is isometric, that:
$$\begin{array}{rcl}
| \gamma^{-1} |_\Sigma & \leqslant & 2L \cdot d_\g(v_0, hg.v_0)\\
& \leqslant & 2L \cdot \big( d_\g(v_0, h .v_0)  + d_\g(h.v_0, hg.v_0) \big)\\
& \leqslant & 2L \cdot \big( d_\g(v_0, h .v_0)  + d_\g(v_0, g.v_0) \big).
\end{array}$$

Finally, by definition of the Schreier graph we deduce that $| \gamma^{-1} |_\Sigma \leqslant 2L \cdot \bigl( \|g\|_{\widehat{\Sigma}} +  \|h\|_{\widehat{\Sigma}} \bigr)$.
Recall that we want to prove that the function $h \mapsto |\alpha(g, h)|_\Sigma$ belongs to ${\rm L}^p(\mathscr{F},{\rm d}h)$.
Since ${\rm Vol}(\mathscr{F},{\rm d}h)=1$, so does the constant function $h \mapsto \|g\|_{\widehat{\Sigma}}$, therefore it remains to prove the lemma below.
\end{proof}

\begin{lem}
The function $h \mapsto \|h\|_{\widehat{\Sigma}}$ belongs to ${\rm L}^p(\mathscr{F},{\rm d}h)$.
\end{lem}

\begin{proof}
Let $h \in \mathscr{F}$.
By construction of the fundamental domain $\mathscr{F}$, there exist $i \geqslant 0$ and $u_i$ in $F_i$, hence in ${\rm Stab}_G(v_i)$, such that $h = u_ig_i^{-1}$.
This implies $h.v_0 = u_i.(g_i^{-1}.v_0) = u_i.v_i = v_i$, and also $(\gamma h).v_0 = \gamma.v_i$ for each $\gamma \in \Gamma$.
Now the explicit form of the quasi-isometry equivalence (\ref{ss - Schreier graphs}) between $(\mathfrak{g},d_\mathfrak{g})$ and $(G, d_{\widehat{\Sigma}})$ implies:

\medskip

\hspace{1cm}$d_\g (v_0, h.v_0) \leqslant \|h\|_{\widehat{\Sigma}} \leqslant d_\g (v_0, h.v_0) +1$,

\smallskip

and

\smallskip

\hspace{1cm}$d_\g (v_0, (\gamma h).v_0) \leqslant \|\gamma h\|_{\widehat{\Sigma}} \leqslant d_\g (v_0, (\gamma h).v_0) +1$.

\medskip

\noindent
Moreover by the choice of $v_i$ in its $\Gamma$-orbit, we have $d_\g (v_0, h.v_0) \leqslant d_\g (v_0, (\gamma h).v_0)$ for any $\gamma \in \Gamma$.
This allows us to put together the above two double inequalities, and to obtain (after forgetting the extreme upper and lower bounds):

\medskip

\hspace{1cm}$(\dagger)$ \quad $\|h\|_{\widehat{\Sigma}} \leqslant \|\gamma h\|_{\widehat{\Sigma}} + 1$.

\medskip

\noindent
for any $h \in \mathscr{F}$ and $\gamma \in \Gamma$.

\medskip

Recall that $p \in [1; +\infty)$ is an integer such that we have a Borel fundamental domain $\Omega$ for which $\displaystyle \int_\Omega \bigl( \| h \|_{\widehat{\Sigma}} \bigr)^p \mathrm{d}h < \infty$.
Since $G = \bigsqcup_{\gamma \in \Gamma} \gamma^{-1} \Omega$ we can write:

\medskip

\hspace{1cm} $\displaystyle \int_\mathscr{F} \bigl( \| h \|_{\widehat{\Sigma}} \bigr)^p \mathrm{d}h
= \sum_{\gamma \in \Gamma}  \, \int_{\mathscr{F} \cap \gamma^{-1}\Omega} \bigl( \| h \|_{\widehat{\Sigma}} \bigr)^p \mathrm{d}h
$.

\medskip

But in view of $(\dagger)$ and of the unimodularity of $G$ (which contains a lattice), we have:

\medskip

\hspace{1cm} $\displaystyle \int_{\mathscr{F} \cap \gamma^{-1}\Omega} \bigl( \| h \|_{\widehat{\Sigma}} \bigr)^p \mathrm{d}h
\leqslant  \int_{\mathscr{F} \cap \gamma^{-1}\Omega} \bigl( \| \gamma h \|_{\widehat{\Sigma}} + 1 \bigr)^p \mathrm{d}h
= \int_{\gamma\mathscr{F} \cap \Omega} \bigl( \| h \|_{\widehat{\Sigma}} + 1 \bigr)^p \mathrm{d}h$,

\medskip

which finally provides

\medskip

\hspace{1cm} $\displaystyle \int_\mathscr{F} \bigl( \| h \|_{\widehat{\Sigma}} \bigr)^p \mathrm{d}h
\leqslant \sum_{\gamma \in \Gamma}  \, \int_{\gamma\mathscr{F} \cap \Omega} \bigl( \| h \|_{\widehat{\Sigma}} + 1 \bigr)^p \mathrm{d}h
= \int_\Omega \bigl( \| h \|_{\widehat{\Sigma}} + 1 \bigr)^p \mathrm{d}h$.

\medskip

The conclusion follows because $\mathscr{F}$ has Haar volume equal to 1 and because by assumption $h \mapsto \|h\|_{\widehat{\Sigma}}$ belongs to ${\rm L}^p(\Omega,{\rm d}h)$.
\end{proof}

\subsection{$p$-integrability of twin building lattices}
\label{ss - sq-int twin buildings}
Let us finish by mentioning the following fact which, using Theorem~\ref{th:InductionCocycle}, allows us to prove the main result of \cite{Remy:Integrable} in a more conceptual way.

\begin{lem}
\label{lem - sq-int twin buildings}
Let $\Gamma$ be a twin building lattice and let $G$ be the product of the automorphism groups of the associated buildings $X_\pm$.
Let $W$ be the Weyl group and $\sum_{n \geqslant 0} c_n t^n$ be the growth series of $W$ with respect to its canonical set of generators $S$, i.e.,  $c_n = \# \{Êw \in W : \ell_S(w) = n \}$.
Let $q_{\rm min}$ denote the minimal order of root groups and assume that $\sum_{n \geqslant 0} c_n q_{\rm min}^{-n} < \infty$.
Then $\Gamma$ admits a fundamental domain $\mathscr{F}$ in $G$, with associated induction cocycle $\alpha_\mathscr{F}$, such that $h \mapsto \alpha_\mathscr{F}(g,h)$ belongs to  ${\rm L}^p(\mathscr{F},{\rm d}h)$ for any $g \in G$ and any $p \in [1;+\infty)$.
\end{lem}

\begin{proof}
We freely use the notation of \ref{ss - generating system} and \cite{Remy:Integrable}.
We denote by $\mathscr{B}_\pm$ the stabilizer of the standard chamber $c_\pm$ in the closure $\overline{\Gamma}^{{\rm Aut}(X_\pm)}$.
By [loc. cit.] there is a fundamental domain $\mathscr{F} = D = \bigsqcup_{w \in W} D_w$ such that ${\rm Vol}(D_w,{\rm d}h) \leqslant q_{\rm min}^{-\ell_S(w)}$.
If we choose the compact generating set $\widehat{\Sigma} = \bigsqcup_{(s_-,s_+) \in S \times S} \mathscr{B}_-s_-\mathscr{B}_- \times \mathscr{B}_+s_+\mathscr{B}_+$, we see that by definition of $D_w$, which his contained in $\mathscr{B}_- \times \mathscr{B}_+w$, we have $\| h \|_{\widehat{\Sigma}} \leqslant  \ell_S(w)$ for any $w \in W \setminus \{Ê1 \}$ and any $h \in D_w$.
Therefore for any $p \in [1;+\infty)$ we have: $\displaystyle \int_\mathscr{F}Ê\bigl( \| h \|_{\widehat{\Sigma}} \bigr)^p  {\rm d}h \leqslant \sum_{n \geqslant 0} n^p c_n q_{\rm min}^{-n}$, from which the conclusion follows.
\end{proof}


\bibliographystyle{amsalpha}
\bibliography{NonDistorsion}
\end{document}